\documentclass[a4paper,12pt]{article}

\usepackage{amsmath}
\usepackage{amssymb}
\usepackage{amsfonts}
\usepackage{amsthm}
\usepackage{mathrsfs}
\usepackage{graphicx}
\usepackage{makeidx}
\usepackage{subcaption}
\usepackage{color}
\usepackage{float}
\usepackage{mdframed}
\usepackage{ulem}
\usepackage{tikz,pgfplots}
\usepackage{hyperref}

\usepackage[left=3cm,right=3cm,top=3cm,bottom=3cm]{geometry}

\newcommand{\R}{\mathbb{R}} 
\newcommand{\N}{\mathbb{N}} 
\newcommand{\cM}{\mathcal{M}} 

\newtheorem{theorem}{Theorem}
\newtheorem{lemma}{Lemma}
\newtheorem{proposition}{Proposition}

\newtheorem{remark}{Remark}

\newcommand{\revise}[1]{{\color{black}#1}}
\newcommand{\didier}[1]{{\color{red} #1}}

\newcommand{\adrien}[1]{{\color{magenta} #1}}

\newcommand{\proscal}[2]{\left\langle#1\:,#2\right\rangle} 

\newcommand{\RR}{\mathbb{R}}
\newcommand{\NN}{\mathbb{N}}

\title{\bf Slow convergence of the moment-SOS hierarchy for an elementary polynomial optimization problem}
\begin{document}

\author{Didier Henrion$^{1,2}$, Adrien Le Franc$^{1}$, Victor Magron$^{1,3}$}

\footnotetext[1]{CNRS; LAAS; Universit\'e de Toulouse, 7 avenue du colonel Roche, F-31400 Toulouse, France. }
\footnotetext[2]{Faculty of Electrical Engineering, Czech Technical University in Prague,
Technick\'a 2, CZ-16626 Prague, Czechia.}
\footnotetext[3]{Institute of Mathematics of Toulouse, 118 route de Narbonne, F-31062 Toulouse, France.}

\date{Draft of \today}

\maketitle

\begin{abstract}
We describe a parametric univariate quadratic optimization problem for which the moment-SOS hierarchy has finite but increasingly slow convergence when the parameter tends to its limit value. We estimate the order of finite convergence as a function of the parameter. 
\end{abstract}

\section{Introduction}

The moment-SOS (sum of squares) hierarchy proposed in \cite{l01} is a powerful approach for solving globally non-convex polynomial optimization problems (POPs) at the price of solving a family of convex semidefinite optimization problems (called moment-SOS relaxations) of increasing size, controlled by an integer, the relaxation order. Under standard assumptions, the resulting sequence of bounds on the global optimum converges asymptotically, i.e., when the relaxation order goes to infinity \cite{hkl20,h23,n23}. We say that {\it convergence is finite} if the bound matches the global optimum at a finite relaxation order, i.e., solving the relaxation actually solves the POP globally. 
In this case, we also say that the {\it relaxation is exact}. 
It is known that convergence is finite generically \cite{nie2014optimality}, which means that POPs for which convergence is not finite
are exceptional. 
In the specific case of minimizing a univariate polynomial of degree $d$ over either the real line, a given closed interval or $[0, \infty)$, finite convergence systematically occurs at the minimal relaxation order $\lceil d/2 \rceil$ as a consequence of \cite{fekete35,ps76} (see also \cite[~\S~2.3]{l09} for a modern exposition) if the interval is described with a proper set of linear/quadratic inequality constraints. 
In addition to finite convergence occurring systematically in such univariate POPs, one can even bound the running time required to compute the resulting SOS decompositions when the input data involve rational coefficients \cite{mag19,mag22}. 

Restricting our attention to POPs on compact sets,
several elementary examples without finite convergence are known. 
%
%
A bivariate example is $\min_{x \in \R^2} (1-x^2_1)(1-x^2_2) \:\mathrm{s.t.}\: 1-x^2_1\geq 0, 1-x^2_2 \geq 0$, see \cite[Prop. 29]{bs23}. Many results on the speed of convergence of the moment-SOS hierarchy are now available, see, e.g., \cite{b22,s22}. However, they are practically not useful since the rates are asymptotic, for very large values of the relaxation order. In practice, when implementing the moment-SOS hierarchy on low-dimensional benchmark POPs, we observe a finite and fast convergence of the hierarchy, see, e.g., the original experiments carried out in \cite{hl03}. To the best of the authors' knowledge, almost nothing is known about the speed of convergence for small relaxation orders, when the convergence is finite.

The contribution of this note is to describe and study the elementary bounded univariate quadratic POP
\[
\begin{array}{rl}
\min_{x \in \R} & x \\
\mathrm{s.t.} &1-x^2 \geq 0 \\
 & x+(1-\varepsilon)x^2 \geq 0 ,\\
\end{array}
\]
parametrized by a scalar $\varepsilon \in [0,1]$, such that the convergence of the moment-SOS hierarchy is finite, but arbitrarily slow, when $\varepsilon$ tends to zero. 
{The feasible set is the segment $[0, 1]$ for every $\varepsilon > 0$. When $\varepsilon=0$, the feasible set is the union of the segment $[0,1]$ and the point $-1$.}
We estimate the order of finite convergence as a function of $\varepsilon$. 

Our contributions can be summarized as follows:
\begin{itemize}

\item Finite convergence holds for all $\varepsilon \in [0, 1]$ 
({Proposition}~\ref{prop:finitecvg}) and there exist steps / threshold values $\varepsilon_d \in [0,1]$
    for the exactness of the relaxation of order $d$
    ({Proposition}~\ref{th:varepsilon_d});

\item 
We provide lower and upper bounds on $\varepsilon_d$ to estimate the order at which finite convergence occurs (Theorem~\ref{th:bounds}).

\end{itemize}

Interestingly, this implies that we are able to generate simple univariate POPs with arbitrarily large {finite} convergence orders.

\section{POP design}
\label{sec:popdesign}
In this section we explain how our parametric univariate quadratic POP is designed.

\subsection{Circle and two lines}


We {introduce} the following toy problem, which turns out to be a bivariate quadratic POP:
\begin{equation}\label{pop}
\begin{array}{rrl}
v^*(\varepsilon) = & \min_{x \in \R^2} & x_1 \\
& \mathrm{s.t.} & x^2_1 + x^2_2 = 1 \\
& & 1-\varepsilon + x_1 - (1-\varepsilon) x_2 \geq 0 \\
& & 1-\varepsilon + x_1 + (1-\varepsilon) x_2 \geq 0, \\
\end{array}
\end{equation}
where $\varepsilon \in [0,1]$ is a given parameter. 
This toy program is motivated by sparse optimization
problems involving the $l_0$ pseudonorm: it was recently observed that
such problems are well-structured when projected onto the unit sphere, as 
$l_0(x) = \mathcal{L}_0(\frac{x}{||x||_2})$, where 
$\mathcal{L}_0$ is a proper lower semicontinuous convex function,
introduced in \cite{chancelier2021hidden}.
Consequently, such problems can be reformulated as convex programs
over the unit sphere, where $\mathcal{L}_0$ --- hence $l_0$ --- can be approximated
as closely as desired by a polyhedral function \cite{le2022}. 

{For every $\varepsilon > 0$,} the non-convex feasible set of~\eqref{pop} is the half circle $\{x \in \R^2 : x^2_1+x^2_2 = 1, x_1 \geq 0\}$. 
{The first and second affine inequality constraints are tight for $x \in \{(-1+\varepsilon,0), (0,1)\}$ and $x \in \{(-1+\varepsilon,0), (0,-1)\}$, respectively.}
Geometrically it follows that $v^*(\varepsilon) = 0$ for all $\varepsilon \in (0,1]$ and $v^*(0)=-1$, so that the value function is lower semi-continuous, with a discontinuity at $0$. 

\subsection{Circle and parabola}
\label{sec:poppara}
Let us replace the two affine constraints in POP \eqref{pop} with a parabolic inequality constraint {that is tight} for $x \in \{(-1+\varepsilon,0),(0,1),(0,-1)\}$:
\begin{equation}\label{poppara}
\begin{array}{rrl}
v^*(\varepsilon) = & \min_{x \in \R^2} & x_1 \\
& \mathrm{s.t.} & x^2_1 + x^2_2 = 1 \\
& & 1-\varepsilon + x_1 - (1-\varepsilon) x^2_2 \geq 0, \\
\end{array}
\end{equation}
where $\varepsilon \in [0,1]$ is a given parameter. 
The value function $v^*(\varepsilon)$ is unchanged. 

\if{
\begin{table}
\begin{center}
\begin{tabular}{r|rrrr|r}
$\log1/\varepsilon$&$v_1(\varepsilon)$&$v_2(\varepsilon)$&$v_3(\varepsilon)$&$v_4(\varepsilon)$ & $v^*(\varepsilon)$\\ \hline
1 & $-6.3212\cdot10^{-1}$ &0 & 0 & 0 & 0 \\
2 & $-8.6466\cdot10^{-1}$ &0 & 0 & 0 & 0 \\
3 & $-9.5021\cdot10^{-1}$ &$-1.4794\cdot10^{-1}$ & 0 & 0 & 0 \\
4 & $-9.8168\cdot10^{-1}$ &$-4.5993\cdot10^{-1}$ & 0 & 0 & 0 \\
5 & $-9.9326\cdot10^{-1}$ &$-7.3383\cdot10^{-1}$ & 0 & 0 & 0 \\
6 & $-9.9752\cdot10^{-1}$ &$-8.8837\cdot10^{-1}$ & $-5.1426\cdot10^{-2}$  & 0 & 0 \\
7 & $-9.9909\cdot10^{-1}$ &$-9.5672\cdot10^{-1}$ & $-2.9832\cdot10^{-1}$ & 0 & 0 \\
8 & $-9.9966\cdot10^{-1}$ &$-9.8376\cdot10^{-1}$ &$-6.0489\cdot10^{-1}$ & 0 & 0  \\
9 & $-9.9988\cdot10^{-1}$ &$-9.9398\cdot10^{-1}$ &$-8.2055\cdot10^{-1}$  & $-1.9947\cdot10^{-3}$ & 0\\
$\infty$ & -1 & -1 & -1 & -1 & -1
\end{tabular}
\caption{Lower bounds $v_d(\varepsilon)$ on the value $v^*(\varepsilon)$ of POP \eqref{poppara} obtained with the moment-SOS hierarchy, for increasing relaxation orders $d$ and different values of $\varepsilon$.}
\label{tab:poppara}
\end{center}
\end{table}

On Table \ref{tab:poppara} we report the values of the lower bounds $v_d$ on the value $v^*$ of POP \eqref{poppara} obtained with the moment-SOS hierarchy and the semidefinite solver SeDuMi, for increasing relaxation orders $d=1,\ldots,4$ and different values of $\log1/\varepsilon$. We report 0 when the bound is less than $10^{-7}$ in absolute value.
}\fi


\begin{table}
\scriptsize
{
\hspace*{-0.8cm}
\begin{tabular}{r|rrrrrrrr|r}
$\log \varepsilon^{-1}$ & $v_1(\varepsilon)$ & $v_2(\varepsilon)$ & $v_3(\varepsilon)$ & $v_4(\varepsilon)$ & $v_5(\varepsilon)$ & $v_6(\varepsilon)$ & $v_7(\varepsilon)$ & $v_8(\varepsilon)$ & $v^*(\varepsilon)$\\ 
\hline
1 & $-0.900000$ & $-0.014972$ & 0 & 0 & 0 & 0 & 0 & 0  & 0\\
2 & $-0.990000$ & $-0.639195$ & 0 & 0 & 0 & 0 & 0 & 0 & 0\\
3 & $-0.999000$ & $-0.952678$ & $-0.270070$ & 0 & 0 & 0 & 0 & 0 & 0\\
4 & $-0.999900$ & $-0.995117$ & $-0.850752$ & $-0.014093$ & 0 & 0 & 0 & 0 & 0\\
5 & $-0.999990$ & $-0.999510$ & $-0.983399$ & $-0.601522$ & 0 & 0 & 0 & 0 & 0\\
6 & $-0.999999$ & $-0.999951$ & $-0.998321$ & $-0.945228$ & $-0.228194$ & 0 & 0 & 0 & 0\\
7 & $-1.000000$ & $-0.999995$ & $-0.999832$ & $-0.994312$ & $-0.830694$ & $-0.004678$ & 0 & 0 & 0 \\
8 & $-1.000000$ & $-1.000000$ & $-0.999983$ & $-0.999429$ & $-0.980874$  & $-0.560890$ & 0 & 0 & 0\\
9 & $-1.000000$ & $-1.000000$ & $-0.999998$ & $-0.999943$ & $-0.998062$ & $-0.937191$ & $-0.188023$ & 0 & 0\\
$\infty$ & -1 & -1 & -1 & -1 & -1 & -1 & -1 & -1 & -1
\end{tabular}
}
\caption{Lower bounds $v_d(\varepsilon)$ on the value $v^*(\varepsilon)$ of POP \eqref{poppara} obtained with the moment-SOS hierarchy, for increasing relaxation orders $d$ and different values of $\varepsilon$.}
\label{tab:poppara}
\end{table}
{
On Table \ref{tab:poppara} we report the values of the lower bounds $v_d$ on the value $v^*$ of POP \eqref{poppara} obtained with the moment-SOS hierarchy and the high-precision semidefinite solvers ClusteredLowRankSolver \cite{ll22} and Loraine \cite{hks23} (both output similar results), for increasing relaxation orders $d=1,\ldots,8$ and different values of $\log1/\varepsilon$. 
We report 0 when the bound is less than $10^{-60}$ in absolute value. 

We observe that the hierarchy converges to the optimal value at finite relaxation order, but convergence is slower when $\varepsilon$ tends to zero. 
A similar table can be obtained with standard semidefinite solvers, e.g., Mosek \cite{aa20} or SeDuMi \cite{sturm99},  implemented in double floating-point precision, but the output values happen to be very inaccurate. 
For instance the first two entries of the first row obtained with either Mosek or SeDuMi are approximately equal to $-0.63$ and $0$, instead of $-0.90$ and $-0.149$, respectively. 
}

\begin{figure}
\begin{center}
\includegraphics[width=0.7\textwidth]{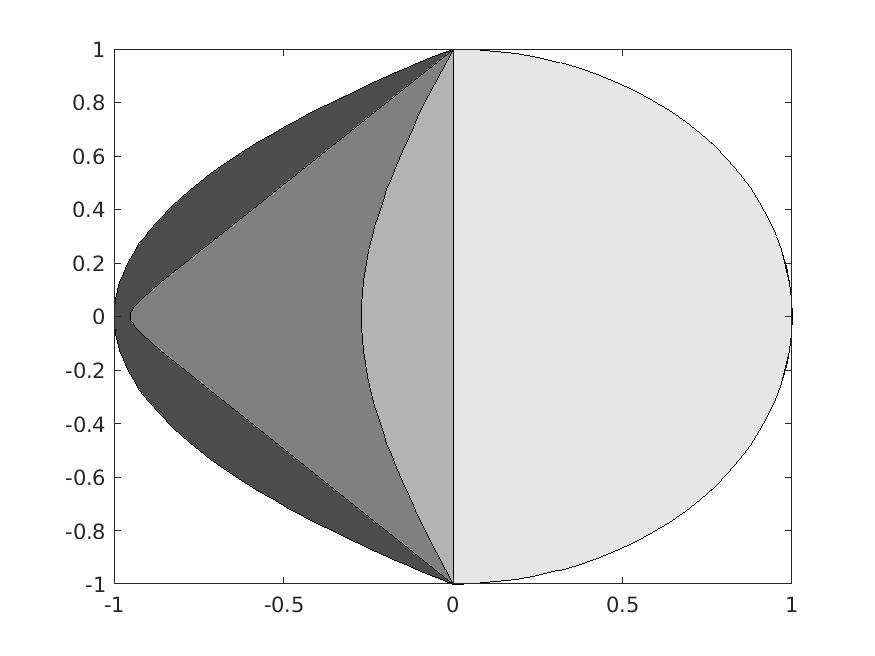}
\caption{Nested projections on the first degree moments of the moment relaxations for increasing orders $d=1,\ldots,4$, from dark to clear gray.}
\label{fig:conv2}
\end{center}
\end{figure}

On Figure \ref{fig:conv2} we represent the nested projections on the first degree moments of the moment relaxations for orders $d=1,\ldots,4$, for the case $\varepsilon=3\cdot10^{-3}$. They are outer approximations of the convex hull of the feasible set, namely the half disk $\{x \in \R^2 : x^2_1+x^2_2 \leq 1, x_1 \geq 0\}$.
While the relaxation of order $4$ is the half disk, we observe that the lower order relaxations are not tight.

\subsection{Univariate parabolic reduction}


Letting $x^2_2=1-x^2_1$, we can reformulate bivariate POP \eqref{poppara} as the univariate POP
\begin{equation}\label{univpara}
\begin{array}{rcrl}
v^*(\varepsilon) & = & \min_{x \in \R} & x \\
&&\mathrm{s.t.} &1-x^2 \geq 0 \\
&&& x+(1-\varepsilon)x^2 \geq 0, \\
\end{array}
\end{equation}
where the constraint $1-x^2\geq 0$ keeps track of positivity of $x^2_2$.
The feasible set of POP \eqref{univpara} is the segment $[0,1]$ if $\varepsilon \in (0,1]$. If $\varepsilon=0$, the feasible set is the non-convex union of $[0,1]$ and $\{-1\}$.
Next we recall background on real algebraic geometry and the moment-SOS hierarchy of semidefinite relaxations to approximate as closely as desired the solution of POP \eqref{univpara}. 


Let $\R[x]_d$ denote the vector space of polynomials of $x$
of degree up to $d$, and let $\Sigma[x]_{2d} \subset \R[x]_{2d}$ denote the convex cone of polynomials that can be expressed as sums of squares (SOS) of polynomials of degree up to $d$.
{
For more details on the definitions and results mentioned hereafter, we refer the interested reader to \cite{marshall08}. 
Given a set of polynomials $S = \{g_1, \dots, g_m\} \subset \R[x]$ and an integer $d$, let $\cM(S)_{2 d}$ be the \textit{truncated quadratic module} of degree $2d$ generated by $S$, i.e., $\cM(S)_{2d} := \{\sum_{j=0}^m \sigma_j g_j :  \sigma_j \in \Sigma[x]_{2d},
\sigma_j g_j \in \R[x]_{2d}  \}$ with $g_0:=1$.
The \textit{quadratic module} $\cM(S)$ generated by $S$ is $\cM(S) := \bigcup_{d \in \N} \cM(S)_{2d}$. 
It is clear that every polynomial from $\cM(S)$ is nonnegative on $K(S) := \{x \in \R : g_1(x) \geq 0,\dots, g_m(x) \geq 0 \}$. 
We say that $\cM(S)$ is \textit{saturated} if it coincides with the full cone of polynomials  nonnegative on $K(S)$. 
We say that $\cM(S)$ is \textit{stable} if, for every degree $d$, there exists $k=k(d)$ such that every polynomial of degree $d$ in $\cM(S)$ can be represented in $\cM(S)_{2k}$.
}
Given $p(x)=\sum_{a=0}^{2d} p_a x^a \in \R[x]_{2d}$ and $y \in \R^{2d+1}$, define the linear functional $\ell_y$ such that $\ell_y(p(x))=\sum_{a=0}^{2d} p_a y_a$.
The relaxation of order $d$  of the moment-SOS hierarchy for POP \eqref{univpara}  consists of solving the primal moment problem
\begin{equation}
    \begin{array}{rcll}
\text{mom}_d(\varepsilon) & = & \inf_{y \in \R^{2d+1}} & \ell_y(x) \\
& &   \mathrm{s.t.} & \ell_y(p(x)) \geq 0, \forall p \in \Sigma[x]_{2d} \\
& & & \ell_y((1-x^2)p(x)) \geq 0, \forall p \in \Sigma[x]_{2(d-1)} \\
& & & \ell_y((x-(1-\varepsilon)x^2)p(x)) \geq 0, \forall p \in \Sigma[x]_{2(d-1)} \\
& & & \ell_y(1)=1 , 
    \end{array}
\end{equation}
and the dual SOS problem
\begin{equation}
    \label{eq:SOS_d}
    \begin{array}{rcrl}
\text{sos}_d(\varepsilon) & = & \sup_{q,r,s,v} & v \\  
& &  \mathrm{s.t.} & x-v = q(x) + r(x) (1-x^2) + s(x) (x+(1-\varepsilon)x^2) \\
& & & q \in \Sigma[x]_{2d}, r \in \Sigma[x]_{2(d-1)}, s \in \Sigma[x]_{2(d-1)}, v \in \R.
   \end{array}
   \end{equation}

{
Using the above notations, the equality constraint from \eqref{eq:SOS_d} can be rewritten as $x-v \in \cM(S_{\varepsilon})_{2d}$ with $S_{\varepsilon} := \{1-x^2,x+(1-\varepsilon)x^2\}$. 
}

Thanks to the constraint $1-x^2 \geq 0$, these are semidefinite optimization problems without duality gap \cite{jh16}:
$$v_d(\varepsilon):=\text{mom}_d(\varepsilon)=\text{sos}_d(\varepsilon).$$
We say that the relaxation of order $d$ is exact when  $v_d(\varepsilon)= v^\star(\varepsilon)$.  
A sufficient condition for exactness is when this latter value is attained in the SOS dual, namely
\begin{equation}\label{sos}
x-v^\star(\varepsilon) = q(x) + r(x) (1-x^2) + s(x) (x+(1-\varepsilon)x^2) ,
\end{equation}
for some SOS polynomials $q,r,s$ of appropriate degrees.

\if{
\adrien{I think Nie's definition of exactness is rather that 
$v_d(\varepsilon) = v^*(\varepsilon)$, which is weaker than~\eqref{sos}
as there are examples with finite convergence
but no maximizer: think of $x-v = q(x) - r(x)x^2$.
I don't know if his result ensures
that the above holds.
If it does, we might not need Lemma~\ref{le:compactness} 
but we should keep the part of its proof 
which is reused to show Corollary~\ref{th:varepsilon_d}.}
}\fi

\if{
We will prove that when $\varepsilon \in (0,1)$, it holds $v^\star(\varepsilon)=0$ and letting $x=0$ in \eqref{sos} implies that $q(0)+r(0)=0$ and since both $q$ and $r$ are SOS this implies $q(0)=r(0)=0$ and hence both contain the factor $x^2$, so equation \eqref{sos} can be rewritten as
\begin{equation}\label{sosred}
x = q(x) x^2 + r(x)x^2(1-x^2)+s(x)(x+(1-\varepsilon)x^2)
\end{equation}
\didier{for some SOS polynomials $q$, $r$ resp. $s$ of degree $2(d-1)$, $2(d-2)$ resp. $2(d-1)$.}
Factoring out $x$ and letting $x=0$ implies that $s(0)=1$.
}\fi

The values of the lower bounds $v_d(\varepsilon)$ on the value $v^*(\varepsilon)$ of POP \eqref{univpara} obtained with the moment-SOS hierarchy are the same as the one reported in Table \ref{tab:poppara}.

\section{Analysis}

\begin{figure}
\begin{center}
\includegraphics[width=0.3\textwidth]{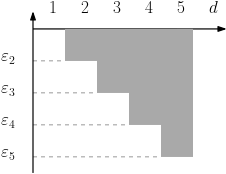}
\caption{Staircase pattern of the moment-SOS hierarchy: the gray region corresponds to values of relaxation order $d$ and parameter $\varepsilon$ for which the relaxation is exact. The steps correspond to threshold values $\varepsilon_2 \geq \varepsilon_3 \geq \revise{\varepsilon_4  \geq \,} \varepsilon_5$ for each degree.}
\label{fig:staircase}
\end{center}
\end{figure}

The aim of this section is to explain analytically how the numerical values of Table \ref{tab:poppara} follow the staircase pattern of Figure \ref{fig:staircase}.
The gray region on Figure \ref{fig:staircase} corresponds to values of $d$ and $\varepsilon$ for which the relaxation is exact, i.e., $v_d(\varepsilon)=v^*(\varepsilon)$.
The steps of the staircase correspond to limit values $\varepsilon_d$ such that for all $\varepsilon \geq \varepsilon_d$, the relaxation of order $d$ is exact.

{Proposition} \ref{prop:finitecvg} shows that we indeed have a staircase geometry. {Proposition} \ref{th:varepsilon_d} shows that the staircase has genuine steps, i.e., the limit values $\varepsilon_d$ can be attained. Finally, Theorem \ref{th:bounds} gives explicit lower and upper bounds on $\varepsilon_d$ as functions of $d$, to quantify the slope of the staircase.

\subsection{Finite convergence}\label{sec:finite}


{
\begin{proposition}
\label{prop:finitecvg}
For any $\varepsilon \in [0,1]$,
the moment-SOS hierarchy converges in a finite number of steps.  In particular $v_d(0)=-1$ for all orders $d$, and for all $\varepsilon \in (0, 1]$ one has  $v_d(\varepsilon)=0$ for a finite relaxation order $d$  depending on $\varepsilon$.
\end{proposition}
}
\begin{proof}
    If $\varepsilon = 0$, Equation \eqref{sos} admits the elementary solution $v_1(0)=v^*(0)=-1$ for $q=0$, $r=1$ and $s=1$, so the first relaxation is exact.

{
We now consider $\varepsilon \in (0,1]$, $S_{\varepsilon} = \{g, g_{\varepsilon} \}$ with $g=1-x^2$ and $g_{\varepsilon} = x + (1-\varepsilon)x^2$. 
POP~\eqref{univpara} has a unique local and global minimizer $x^* = 0$ over the feasibility set  $K(S_{\varepsilon})=[0,1]$. 
As suggested by an anonymous referee, 
we now prove that the quadratic module  $\cM(S_{\varepsilon})$ is saturated. 
The end points of the closed interval $K(S_{\varepsilon})= [0, 1]$ satisfy $g_{\varepsilon}(0) = 0 $, $g_{\varepsilon}'(0) = 1 > 0$, $g(1)=0$ and $g'(1) = -1 < 0$. 
Then it follows from Theorem 9.3.1 and Theorem 9.3.3 in \cite{marshall08} that $\cM(S_{\varepsilon})$ is saturated, and in particular $x$ belongs to $\cM(S_{\varepsilon})$. 
In other words, the moment-SOS hierarchy converges in finitely many steps for POP~\eqref{univpara}. 
}        
\end{proof}
{
\begin{remark}
For the case $\varepsilon \in (0,1]$ we present an alternative proof of {Proposition}~\ref{prop:finitecvg} based on the results from \cite{nie2014optimality}. 
We first check that the so-called \textit{constraint qualification} holds at each local minimizer, i.e., that the corresponding gradient evaluations of the polynomials involved in the active constraints are linearly independent.  
Since $g(0) = 1$ and $g_{\varepsilon} = 0$, there is only one active constraint at the global minimizer $x^*=0$ so the constraints are qualified. 

We now introduce the Lagrangian
    \begin{equation*}
        L(x, \mu) = x - \mu_1(1-x^2) - \mu_2(x + (1-\varepsilon) x^2) \,,
    \end{equation*}
    which gives the following first-order necessary conditions of optimality:
    \begin{equation*}
        \begin{cases}
            1 + 2\mu_1x - \mu_2 - 2\mu_2(1-\varepsilon)x = 0 \,, \\
            \mu_1(1-x^2) = 0 \,, \\
            \mu_2(x + (1-\varepsilon)x^2) = 0 \,, \\
            x \in [0,1] \, , \mu \geq 0 \,.
        \end{cases}
    \end{equation*}

    We deduce that there exists a unique pair of optimal
    multipliers $(\mu_1, \mu_2) = (0, 1)$. Since moreover, 
    $x + (1-\varepsilon)x^2 \geq 0$ is the only active 
    constraint at $x^*$ with associated multiplier
    $\mu_2 > 0$, strict complementarity is satisfied.

    Lastly, the Jacobian of the active constraint
    at $x^*$ is $J(x^*) = 1$, so that its null space in $\RR$
    is reduced to $\{0\}$. With the second-order derivative
    $\nabla^2_x L(x^*,\mu^*) = - 2(1-\varepsilon)$,
    we check that the second-order necessary condition of optimality
    \begin{equation*}
        x^T \nabla^2_x L(x^*,\mu^*) x \geq 0 \, , \forall x \in \{0\}
    \end{equation*}
    is satisfied, and so is the second-order sufficient condition of optimality
    \begin{equation*}
        x^T \nabla^2_x L(x^*,\mu^*) x > 0 \,,   \text{ for all} \ 0 \neq x \in \{0\}  \,.
    \end{equation*}
    We conclude that the moment-SOS hierarchy converges in finitely
    many steps for POP~\eqref{univpara}, from \cite[Theorem 1.1]{nie2014optimality}.    
\end{remark}
Note that finite convergence does not necessarily come together with the existence of maximizers in the SOS hierarchy. 
A simple example in the univariate case is $\min_{x \in \R} x \:\mathrm{s.t.}\: -x^2 \geq 0$, see \cite[Ex. 1.3.4]{b22} or \cite[\S 2.5.2]{n23}, where finite convergence occurs at the first order but the corresponding SOS optimization problem has no maximizer. }

We now prove that, beyond finite convergence,
there always exist maximizers in~\eqref{eq:SOS_d}. 
Let $\varepsilon \in [0, 1]$ and $d \in \NN^*$. 
We notice that, as $v_1(\varepsilon) \leq v_d(\varepsilon) \leq v^*(\varepsilon)$, we must have $v_d(\varepsilon) \in [-1, 0]$.
\if{
Let $\varepsilon \in [0, 1]$ and $d \in \NN^*$. 
We notice that, as $v_1(\varepsilon) \leq v_d(\varepsilon) \leq v^*(\varepsilon)$, we must have $v_d(\varepsilon) \in [-1, 0]$.
So let us denote by $S^{[-1,0]}_d(\varepsilon)$ the feasible set of the SOS relaxation~\eqref{eq:SOS_d} with the additional constraint $v \in [-1,0]$. The supremum in~\eqref{eq:SOS_d} can be equivalently 
searched over $S^{[-1,0]}_d(\varepsilon)$.
}\fi
{
\begin{lemma}
\label{le:compactness}
The set $\cM(S_{\varepsilon})_{2 d}$ is compact and 
there always exists a solution $(q, r, s, v)$ maximizing \eqref{eq:SOS_d}. 
\end{lemma}
}
\begin{proof}
As suggested by an anonymous referee, 
for $\varepsilon \in [0,1]$, the set $K(S_{\varepsilon})$ has nonempty interior, thus Lemma 4.1.4 and Theorem 10.5.1 from \cite{marshall08} imply that for each $d\in \N$ the truncated quadratic module $\cM(S_{\varepsilon})_{2 d}$ is closed. 

We now concentrate on proving that it is bounded.
	Let {$(q,r,s,v)$  be such that 
	$x-v = q(x) + r(x) (1-x^2) + s(x) (x+(1-\varepsilon)x^2) \in \cM(S_{\varepsilon})_{2 d}$. 
    By the above remark one has $v\in [-1, 0]$. }
	Let $I = [0.1, 0.9]$,
	we recall that, for any $n \in \NN$,
	$||p||_I = \max_{x \in I} |p(x)|$
	is a norm on $\RR[x]_n$.
	Now,	
	for any $\sigma \in \{q, r, s\}$
	and its associated constraint
	$g_{\varepsilon} \in \{1, 1-x^2, x+(1-\varepsilon)x^2\}$,
	it holds that
	\begin{equation*}
		x-v - \sigma(x)g_{\varepsilon}(x) \geq 0
		\,, \forall x \in I 
	\end{equation*}
	with $g_{\varepsilon}(x) > 0$ over $I$.
	We deduce that
	\begin{equation*}
		\sigma(x) \leq \max_{x \in I, v \in [-1,0], \varepsilon \in [0,1]} \ \frac{x-v}{g_{\varepsilon}(x)}
		\,, \forall x \in I \,
	\end{equation*}
	so that, as $\sigma \geq 0$,
	$||\sigma||_I \leq U$ for any upper bound $U > 0$
	larger than the right-hand side in the above inequality.
	In particular, let us set $U$ so that $||\sigma||_I \leq U$
	for all $\sigma \in \{q, r, s\}$. 
    {
    Since $v \in [-1, 0]$,
	we conclude that $\cM(S_{\varepsilon})_{2 d}$}
	is bounded, and thus compact.
\end{proof}

We now consider the step value

\begin{equation}
\label{eq:varepsilon_d}
\begin{array}{rcrl}
\varepsilon_d & := & \inf_{\varepsilon,q,r,s} & \varepsilon \\  
& &  \mathrm{s.t.} & x = q(x) + r(x) (1-x^2) + s(x) (x+(1-\varepsilon)x^2) \\
& & & \varepsilon \in [0,1], q \in \Sigma[x]_{2d}, r \in \Sigma[x]_{2(d-1)}, s \in \Sigma[x]_{2(d-1)}.
\end{array}
\end{equation}
{
\begin{proposition}
\label{th:varepsilon_d}
The infimum in~\eqref{eq:varepsilon_d} is attained.
\end{proposition}
\begin{proof}
Let us asumme that $(\varepsilon, q,r,s)$ satisfies~\eqref{eq:varepsilon_d}, so in particular $\varepsilon \in [0,1]$ and $x \in \cM(S_{\varepsilon})_{2 d}$. 
We deduce from Lemma~\ref{le:compactness} that the feasible set of~\eqref{eq:varepsilon_d} is compact, yielding the desired result. 
\end{proof}
}



Letting $x=0$ in \eqref{eq:varepsilon_d} implies that $q(0)+r(0)=0$ and since both $q$ and $r$ are SOS this implies $q(0)=r(0)=0$ and hence both contain the factor $x^2$, so the equality constraint can be rewritten as
\begin{equation}\label{sosred}
x = {\tilde{q}}(x) x^2 + {\tilde{r}}(x)x^2(1-x^2)+s(x)(x+(1-\varepsilon)x^2) ,
\end{equation}
for ${\tilde{q},\revise{s}} \in \Sigma[x]_{2 (d-1)}$, $\revise{\tilde{r}} \in  \Sigma[x]_{2 (d-2)}$.
Factoring out $x$ and letting $x=0$ implies that $s(0)=1$.
Let
\[
\begin{array}{rcl}
{\cM^{\star}(S_{\varepsilon})_{2 d}}& :=& \{{(\tilde{q}, \tilde{r}, s )} : 
\tilde{q},s \in \Sigma[x]_{2(d-1)}, 
{\tilde{r}} \in \Sigma[x]_{2(d-2)}, \\
& &
\quad x - v^\star(\varepsilon)=  {\tilde{q}}(x) x^2 + {\tilde{r}}(x)x^2(1-x^2)+s(x)(x+(1-\varepsilon)x^2) \}
\end{array}
\]
denote the feasible set of the dual SOS problem when the relaxation is exact.

{Proposition} \ref{th:varepsilon_d} implies that for every degree $d$, there is a minimal value of $\varepsilon$, denoted $\varepsilon_d$, such that {$\cM^{\star}(S_{\varepsilon})_{2 d}$} is non-empty. 
In the sequel, we provide lower and upper bounds estimates on $\varepsilon_d$. 
Section \ref{sec:lower} is dedicated to lower relaxation orders, namely $d\in\{1,2,3\}$, while Section \ref{sec:higher} focuses on arbitrary high relaxation orders.

\subsection{Lower relaxations}
\label{sec:lower}

{
\begin{lemma}
\label{lemma:cMstar}
If $\cM^{\star}(S_{\varepsilon})_{2 d}$ is non-empty then
$\cM^{\star}(S_{\varepsilon+\delta})_{2 d}$ is non-empty for all $\delta\geq 0$. 
\end{lemma}

\begin{proof}
Given $(\revise{\tilde{q}, \tilde{r}}, s) \in \cM^{\star}(S_{\varepsilon})_{2 d}$,
observe that $(\revise{\tilde{q}}+\delta s x^2, \revise{\tilde{r}}, s) \in \cM^{\star}(S_{\varepsilon+\delta})_{2 d}$ if $\delta \geq 0$.
\end{proof}
}

\subsubsection{First relaxation}

\begin{lemma}\label{exact1}
When $\varepsilon=0$, the first relaxation is exact.
\end{lemma}
\begin{proof}
See the beginning of the proof of {Proposition} \ref{prop:finitecvg}. 
\end{proof}

\begin{lemma}
When $\varepsilon>0$, the first relaxation is never exact.
\end{lemma}

\begin{proof}
The first relaxation always yields the strict lower bound $v_1=\varepsilon-1< v^*(\varepsilon)=0$. Indeed equation \eqref{sos} holds for $q=0$, $r=1-\varepsilon$, $s=1$ and we can use the moment dual to prove that this is optimal: the first  moment relaxation
    for \eqref{univpara}
    writes 
    \begin{equation*}
        \begin{array}{rrl}
    v_1(\varepsilon) = & \min_{y \in \R^2} & y_1 \\
    & \mathrm{s.t.} & 
    \left(\begin{array}{cc}
    1 & y_1 \\
    y_1 & y_2
    \end{array}\right) \succeq 0 \\
    & & 1 - y_2 \geq 0 \\
    & & y_1 + (1-\varepsilon)y_2 \geq 0, \\
\end{array}
    \end{equation*}
    and admits $y = (\varepsilon-1, 1)$
    as a feasible point for any $\varepsilon \in [0,1]$. It follows that
    $v_1(\varepsilon) \leq \varepsilon-1$.
\end{proof}

\subsubsection{Second relaxation}
{
\begin{lemma}
\label{lemma:eps2}
$\varepsilon_2 = 1-\sqrt{3}/2 \approx 0.1340$.
\end{lemma}
}
\begin{proof}
We first prove that $\revise{\varepsilon}_2 \leq 1-\sqrt{3}/2$. 
Denoting ${\tilde{q}}(x)=q_0+q_1x+q_2x^2$, ${\tilde{r}}(x)=r_0$, $s(x)=1+s_1x+s_2x^2$ and identifying like powers of $x$ in equation \eqref{sosred} yields the linear system of equations
\[
\left(\begin{array}{ccc}
1 & 1 & 0 \\
0 & 1-\varepsilon & 1\\
-1 & 0 & 1-\varepsilon
\end{array}\right)
\left(\begin{array}{c}
r_0 \\ s_1 \\ s_2 
\end{array}\right)
+ \left(\begin{array}{c}
q_0 \\ q_1 \\ q_2 
\end{array}\right)=
\left(\begin{array}{c}
-1+\varepsilon \\ 0 \\ 0 
\end{array}\right).
\]
Letting $q_0=q_1=q_2=0$, it holds
\[
\left(\begin{array}{c}
r_0 \\ s_1 \\ s_2 
\end{array}\right)=
\frac{1}{\varepsilon(2-\varepsilon)}
\left(\begin{array}{c}
(1-\varepsilon)^3 \\ -1+\varepsilon \\(1-\varepsilon)^2 
\end{array}\right).
\]
The polynomial $s$ is SOS if and only if $4 s_2 \geq s^2_1$ i.e. $-4\varepsilon^2+8\varepsilon-1=-4(\varepsilon - 1 - \sqrt{3}/2)(\varepsilon - 1 + \sqrt{3}/2)\geq 0$. It follows that the second relaxation is exact if $\varepsilon \geq 1-\sqrt{3}/2$. When $\varepsilon = 1-\sqrt{3}/2$, it holds $\revise{\tilde{r}}(x)=3\sqrt{3}/2$ and $s(x)=(1-\sqrt{3}x)^2$. 
{This implies that $\varepsilon_2 \leq 1 - \sqrt{3}/2$ by Lemma \ref{lemma:cMstar}. }

{
Next we prove that $\varepsilon_2 \geq 1-\sqrt{3}/2$. 
Let $\varepsilon \in (0, 1]$. The second-order
moment relaxation for~\eqref{univpara} 
is given by
\begin{equation}\label{eq:mom2}
\begin{array}{rcrl}
v_2(\varepsilon) & = & \min_{y \in \RR^5} & y_1 \\
&&\mathrm{s.t.} & M_2(y) \succeq 0 ,\,  M_1(gy) \succeq 0 ,\, M_1(g_\varepsilon y) \succeq 0 ,\, y_0 = 1 ,\, 
\end{array}
\end{equation}
where $M_2(y)$ is the second-order moment matrix:
\begin{equation}
M_1(gy) = \begin{pmatrix}
y_0 - y_2 & y_1 - y_3 \\
y_1 - y_3 & y_2 - y_4
\end{pmatrix}
\quad \text{ and } \quad
M_1(g_\varepsilon y) = \begin{pmatrix}
y_1 + (1-\varepsilon)y_2 & y_2 + (1-\varepsilon)y_3 \\
y_2 + (1-\varepsilon)y_3 & y_3 + (1-\varepsilon)y_4
\end{pmatrix} .	
\end{equation}

From \cite{jh16}, due to the ball constraint in~\eqref{univpara},
there exists a solution $y^\star$ to the moment relaxation~\eqref{eq:mom2}
and strong duality holds between~\eqref{eq:mom2} and its Lagrangian dual, the SOS relaxation. 
Moreover, from Lemma~\ref{le:compactness}, there exists a
solution $(t^\star, X^\star)$ to the SOS relaxation. 
Thus, introducing the Lagrangian 
\begin{equation}
	L(y, t, X) = y_1 - t(y_0-1) - \proscal{M_2(y)}{X_0} - \proscal{M_1(gy)}{X_1} 
	- \proscal{M_1(g_\varepsilon y)}{X_2},
\end{equation} 
an optimal primal-dual solution $(y^\star, t^\star, X^\star)$
satisfies the necessary KKT conditions
\begin{equation}
	\label{eq:KKT}
	\begin{cases}
	y^\star \text{ is primal feasible,} \\
	(t^\star, X^\star) \text{ is dual feasible,} \\
	\nabla_yL(y^\star, t^\star, X^\star) = 0, \\
	M_2(y)X_0 = 0 ,\, M_1(gy)X_1 = 0 ,\, 
	M_1(g_\varepsilon y)X_2 = 0 ,\,
	\end{cases}
\end{equation}
which are in fact also sufficient optimality conditions
since the primal problem is convex.

As a first remark, we observe that $\bar{y}=(1, 0, 0, 0, 0)$ is always feasible in~\eqref{eq:mom2}, hence $v_2(\varepsilon) \leq 0$.
Thus, since $v^*(\varepsilon) = 0$, 
the moment-SOS relaxation is exact
if and only if $\bar{y}$ is a solution of~\eqref{eq:mom2}.
We now reformulate as follows: the second relaxation
is exact if and only if there exists a $X$
such that $(\bar{y}, 0, X)$ satisfies the KKT conditions~\eqref{eq:KKT}.

Let us consider~\eqref{eq:KKT}.
Clearly $\bar{y}$ is primal feasible:
$M_{2}(\bar{y})X_0 = 0$ and $M_{1}(g\bar{y})X_1 = 0$
is equivalent to
$X_0^{1,1} = X_0^{1,2} = X_0^{1,3} = 0$ and 
$X_1^{1,1} = X_1^{1,2} = 0$.
Expanding the condition 
$\nabla_yL(\bar{y}, 0, X^\star) = 0$
and introducing $\eta = 1-\varepsilon$,
we conclude that $(\bar{y}, 0, X)$ satisfies the KKT conditions~\eqref{eq:KKT}
if and only if
\begin{equation}
	\begin{cases}
	X_2^{1,1} = 1 \\
	X_0^{2,2} + X_1^{2,2} + \eta X_2^{1,1} + 2X_2^{1,2} = 0 \\
	2X_0^{2,3} + 2\eta X_2^{1,2} + X_2^{2,2} = 0 \\
	X_0^{3,3} - X_1^{2,2} + \eta X_2^{2,2} = 0
	\end{cases} 
	\begin{pmatrix}
	X_0^{2,2} & X_0^{2,3} \\
	X_0^{2,3} & X_0^{3,3}
	\end{pmatrix} \succeq 0 \,, X_1^{2,2} \geq 0
	\,, X_2 \succeq 0 \,,
\end{equation}
which can be further simplified, observing
that $X_1^{2,2} = X_0^{3,3} + \eta X_2^{2,2} \geq 0$, to obtain that $(\bar{y}, 0, X)$ satisfies the KKT conditions~\eqref{eq:KKT}
if and only if
\begin{equation}
\label{eq:SDPfeasibility}
	\begin{cases}
	X_2^{1,1} = 1 \\
	X_0^{2,2} + X_0^{3,3} + \eta X_2^{2,2} + \eta X_2^{1,1} + 2X_2^{1,2} = 0 \\
	2X_0^{2,3} + 2\eta X_2^{1,2} + X_2^{2,2} = 0 	
	\end{cases} 
	\begin{pmatrix}
	X_0^{2,2} & X_0^{2,3} \\
	X_0^{2,3} & X_0^{3,3}
	\end{pmatrix} \succeq 0
	\,, X_2 \succeq 0 \,.
\end{equation}

We introduce the linear operator 
$\mathcal{A} : \mathbb{S}_2 \times \mathbb{S}_2 \to \RR^3$
and the cone $\mathcal{K} = \mathbb{S}_2^+ \times \mathbb{S}_2^+$
which let us write the system~\eqref{eq:SDPfeasibility}
as $\mathcal{A}(Z) = b$ and $Z \in \mathcal{K}$.
We also introduce 
the system $\mathcal{A}^*(u) \in \mathcal{K}^*$
and $\proscal{u}{b} < 0$, which gives
\begin{equation}
	\label{eq:SDPfeasibilityBis}
	\begin{pmatrix}
	u_2 & u_3 \\
	u_3 & u_2
	\end{pmatrix} \succeq 0
	\,,
	\begin{pmatrix}
	u_1 + \eta u_2 & u_2 + \eta u_3 \\
	u_2 + \eta u_3 & \eta u_2 + u_3
	\end{pmatrix} \succeq 0
	\,, u_1 < 0 \,.
\end{equation}
From Farkas lemma for cones, either~\eqref{eq:SDPfeasibility}
is limit-feasible or~\eqref{eq:SDPfeasibilityBis}
has a solution, but not both \cite[Lemma 4.5.6]{gm12}. 

Let $\mathcal{U}_\eta$ denote the set 
defined by~\eqref{eq:SDPfeasibilityBis},
and let us investigate under which conditions it is empty when $\eta \in [0, 1)$.
First, we observe that if $u \in \mathcal{U}_\eta$,
the polynomial
\begin{equation}
	(\eta^2-1)u_2^2 - \eta^2u_3^2 - \eta u_2u_3
	+ u_1(\eta u_2 + u_3)
	=
	\text{det}\begin{pmatrix}
	u_1 + \eta u_2 & u_2 + \eta u_3 \\
	u_2 + \eta u_3 & \eta u_2 + u_3
	\end{pmatrix}
\end{equation}
must have all its terms except $- \eta u_2u_3$
nonpositive: we deduce that $u_3 \leq 0$.
Then, we obtain
\begin{align}
	\label{eq:3Dsys}%
	u \in \mathcal{U}_\eta
	\iff
	\begin{cases}
	u_1 < 0, u_2 \geq 0, u_3 \leq 0 \,, \\
	u_2 \geq -u_3 \,, \\
	u_1 + \eta u_2 \geq 0 \,, \\
	\eta u_2 + u_3 \geq 0 \,, \\
	(u_1 + \eta u_2)(\eta u_2 + u_3) - (u_2 + \eta u_3)^2 \geq 0 \,.
	\end{cases}
\end{align}
As $\mathcal{U}_\eta$ is a cone,
it is nonempty if and only if
its intersection with the $\ell_1$-ball
is nonempty: we can thus enforce $-u_1 + u_2 - u_3 = 1$ in~\eqref{eq:3Dsys}
to claim that
\begin{align}
	\label{eq:2Dsys}%
	\mathcal{U}_\eta \neq \emptyset
	&\iff \exists v \in \RR^2 \,,
	\begin{cases}
		v_2 \geq 0, v_3 \geq 0 \,, \\
		v_2 + v_3 - 1 < 0 \,, \\
		(1+\eta)v_2 + v_3 - 1 \geq 0 \,, \\
		\eta v_2 - v_3 \geq 0 \,, \\
		((1+\eta)v_2 + v_3 - 1)(\eta v_2 - v_3)
		- (v_2 - \eta v_3)^2 \geq 0 \,.
	\end{cases} 
\end{align}
We now introduce
$\bar{v}_2(\eta) = \frac{\eta(1+2\eta)}{2(\eta+1)}$
and $\bar{v}_3(\eta) = 1 - \bar{v}_2(\eta)$
which satisfy, for any $\eta \in (\sqrt{3}/2, 1)$:
\begin{equation}
\begin{cases}
\bar{v}_2(\eta) > 0, \bar{v}_3(\eta) > 0 \,, \\
\bar{v}_2(\eta) + \bar{v}_3(\eta) - 1 = 0 \,, \\
(1+\eta)\bar{v}_2(\eta) + \bar{v}_3(\eta) - 1 > 0 \,, \\
\eta \bar{v}_2(\eta) - \bar{v}_3(\eta) > 0 \,, \\
((1+\eta)\bar{v}_2(\eta) + \bar{v}_3(\eta) - 1)(\eta \bar{v}_2(\eta) - \bar{v}_3(\eta))
- (\bar{v}_2(\eta) - \eta \bar{v}_3(\eta))^2 = \frac{\eta^2(4\eta^2-3)}{4(\eta + 1)} > 0 \,.
\end{cases} 
\end{equation}

We deduce that, for any $\eta \in (\sqrt{3}/2, 1)$,
there exists a small enough perturbation
$\xi_\eta > 0$
such that $(\bar{v}_2(\eta) - \xi_\eta, \bar{v}_3(\eta) - \xi_\eta)$
is a solution of the system in~\eqref{eq:2Dsys}.
It follows that $\mathcal{U}_\eta \neq \emptyset$,
which implies that~\eqref{eq:SDPfeasibility} is infeasible,
and therefore that the second order relaxation is not exact.

We conclude that $\varepsilon_2 \geq 1 - \sqrt{3} / 2$, yielding the desired result.}
\end{proof}

\subsubsection{Third relaxation}

\begin{lemma}
\label{lemma:eps3}
$\varepsilon_3 \leq 1-\frac{\sqrt{3+12 \sqrt{10}\, \sin \! \left(\frac{\arctan \left(3 \sqrt{111}\right)}{3}+\frac{\pi}{6}\right)}}{6} \approx 4.712527\cdot 10^{-3}.$
\end{lemma}
\begin{proof}
Let ${\tilde{q}}(x)=0$, ${\tilde{r}}(x)=r_0+r_1 x +r_2 x^2$, $s(x)=1+s_1x+s_2x^2+s_3x^3+s_4x^4$, identifying like powers of $x$ in equation \eqref{sosred} yields the linear system of equations
\[
\left(\begin{array}{ccc|cccc}
1 & 0 & 0 & 1 & 0 & 0 & 0  \\
0 & 1 & 0 & 1-\varepsilon & 1 & 0 & 0 \\
-1 & 0 & 1 & 0 & 1-\varepsilon & 1 & 0\\
0 & -1 & 0 & 0 & 0 & 1-\varepsilon & 1 \\
0 & 0 & -1 & 0 & 0 & 0 & 1-\varepsilon
\end{array}\right)
\left(\begin{array}{c}
r_0 \\ r_1 \\ r_2 \\ \hline s_1  \\ s_2 \\ s_3 \\ s_4 
\end{array}\right)=
\left(\begin{array}{c}
-1+\varepsilon \\  0 \\ 0 \\ 0  \\ 0
\end{array}\right) ,
\]
whose solutions can be parametrized with $s_3$ and $s_4$ as
\begin{equation}\label{param3}
\left(\begin{array}{c}
r_0 \\ r_1 \\ r_2 \\ \hline s_1  \\ s_2
\end{array}\right)
=
\left(\begin{array}{c}
\frac{(1-\varepsilon)^3}{\varepsilon(2-\varepsilon)} + s_3 \\ (1-\varepsilon)s_3 + s_4 \\ (1-\varepsilon)s_4 \\ \hline \frac{-1+\varepsilon}{\varepsilon(2-\varepsilon)} - s_3 \\ \frac{(1-\varepsilon)^2}{\varepsilon(2-\varepsilon)} - s_4
\end{array}\right).
\end{equation}
The polynomial ${\tilde{r}}$ is SOS if and only if
\[
\left(\begin{array}{cc}
\frac{(1-\varepsilon)^3}{\varepsilon(2-\varepsilon)} + s_3 &
\frac{1-\varepsilon}{2} s_3 + \frac{1}{2}s_4 \\
\frac{1-\varepsilon}{2} s_3 + \frac{1}{2}s_4 & 
(1-\varepsilon)s_4
\end{array}\right) \succeq 0 ,
\]
and the polynomial $s$ is SOS if and only if
\[
\left(\begin{array}{ccc}
1 & \frac{-1+\varepsilon}{2\varepsilon(2-\varepsilon)} - \frac{1}{2}s_3 & z \\
\frac{-1+\varepsilon}{2\varepsilon(2-\varepsilon)} - \frac{1}{2}s_3 & 
\frac{(1-\varepsilon)^2}{\varepsilon(2-\varepsilon)} - s_4 -2z & \frac{1}{2}s_3 \\
z & \frac{1}{2}s_3 & s_4
\end{array}\right) \succeq 0 ,
\]
for some $z \in \R$.



Assume that $s(x)=(1+ax+bx^2)^2$ and ${\tilde{r}}(x)=(c+dx)^2$ for $a,b,c,d \in \R$, i.e., the above positive semidefinite matrices have rank one. So let us express system \eqref{param3} in terms of the coefficients $a,b,c,d$:
\[
\left(\begin{array}{c}
c^2 \\ 2cd \\ d^2 \\ \hline 2a  \\ 2b+a^2
\end{array}\right)
=
\left(\begin{array}{c}
\frac{(1-\varepsilon)^3}{\varepsilon(2-\varepsilon)} + 2ab \\ (1-\varepsilon)2ab + b^2 \\ (1-\varepsilon)b^2 \\ \hline \frac{-1+\varepsilon}{\varepsilon(2-\varepsilon)} - 2ab \\ \frac{(1-\varepsilon)^2}{\varepsilon(2-\varepsilon)} - b^2
\end{array}\right).
\]
Clearing up denominators, this is a system of 5 polynomial equations in 5 unknowns $a,b,c,d,\varepsilon$. Using Maple's Groebner basis engine we can eliminate $a,b,c,d$ and obtain the following polynomial
\[
(64\varepsilon^6 - 384\varepsilon^5 + 944\varepsilon^4 - 1216\varepsilon^3 + 812\varepsilon^2 - 216\varepsilon + 1)(4\varepsilon^2 - 8\varepsilon + 1)^5(\varepsilon - 1)^8,
\]
vanishing at each solution. The first factor is 
a degree 6 polynomial with 4 real roots. The Galois group of this polynomial allows the roots to be expressed with radicals. The smallest positive real root is
\[
1-\sqrt{ \frac{1}{12} + \frac{5}{3}\alpha^{-\frac{1}{3}} + \frac{1}{6}\alpha^{\frac{1}{3}}},\quad \alpha = -1+3i\sqrt{111} ,
\]
which can be expressed with trigonometric functions as in the statement of the lemma. For this value of $\varepsilon$ we have $a \approx -6.9296$, $b \approx 6.6375$, $c \approx -3.5866$, $d \approx 6.6219$. 
{The desired result follows from Lemma \ref{lemma:cMstar}. }
\end{proof}


\subsection{Higher relaxations}
\label{sec:higher}
The aim of this section is to provide quantitative enclosures of $\varepsilon_d$ for arbitrary relaxation orders $d$. 
Our main result is as follows. 
\begin{theorem}
\label{th:bounds}
For all $d \in \NN$, one has 
$(1+2d (4 e)^{2d})^{-1}  \leq \varepsilon_{d+1} \leq 4^{-d}$.
\end{theorem}

The upper bound follows from Proposition \ref{prop:Paulynomial} while the lower bound follows from Proposition \ref{prop:max_finite}. 
The remaining part of this section is dedicated to proving both propositions. 

We start with a preliminary discussion to provide insights to the reader. 
Notice that if \eqref{sos} holds with $v^\star(\varepsilon) = 0$, namely if for all $x \in \R$ 
\begin{equation}
\label{sos0}
x = q(x) + r(x) (1-x^2) + s(x) (x+(1-\varepsilon)x^2) ,
\end{equation}
for some $q \in \Sigma[x]_{2d}, r,s \in \Sigma[x]_{2(d-1)}$, then $s$ satisfies the following inequality  
\begin{equation}\label{ineq}
x-s(x)(x+(1-\varepsilon)x^2) \geq 0, \quad \text{ for all } x \in [-1,1]. 
\end{equation}
{Let us recall at this point that all nonnegative univariate polynomials with real coefficients are sums of squares of polynomials with real coefficients. }
Assume that there is a nonnegative polynomial $s \in \R[x]_{2(d-1)}$, or equivalently $s \in \Sigma[x]_{2(d-1)}$, satisfying \eqref{ineq}. 
Then it is well known from \cite{fekete35} that there exist $q \in \Sigma[x]_{2d}$ and $r \in \Sigma[x]_{2(d-1)}$ satisfying \eqref{sos0}. 
Therefore we shall restrict our attention to nonnegative polynomials $s$ satisfying \eqref{ineq}. 

\def\vareps{0.5}
\def\inveps{2}
\def\onehalfeps{0.666}

\begin{figure}
\begin{center}
\begin{tikzpicture}[xscale=3,yscale=2]
  \draw (-0.2,-0.2) node {$0$};
  \draw (-1,-0.2) node {$-1$};  \draw (-1,2pt) -- (-1,-2pt); \draw (1,-0.2) node {$1$};  \draw (1,2pt) -- (1,-2pt); \draw (-0.2,1) node {$1$};  \draw (-2pt,1) -- (2pt,1); 
  \draw (-0.2,\onehalfeps) node {$\frac{1}{2-\varepsilon}$};  \draw (-2pt,\onehalfeps) -- (2pt,\onehalfeps); 
  \draw (-0.2,\inveps) node {$\frac{1}{\varepsilon}$};  \draw (-2pt,\inveps) -- (2pt,\inveps); 
  \draw[dashed] (-2pt,\onehalfeps) -- (1,\onehalfeps); \draw[dashed] (1,-2pt) -- (1,\onehalfeps); 
  \draw[dashed] (-2pt,\inveps) -- (-1,\inveps); \draw[dashed] (-1,-2pt) -- (-1,\inveps);

  \draw[->] (-1.5,0) -- (1.5,0) node[right] {$x$};
  \draw[->] (0,-0.2) -- (0,2.2) node[above] {};

  \draw[color=black,domain=-1:1]    plot (\x,{1/(\x*(1-\vareps) + 1)});             
  \coordinate [label=$x \mapsto \frac{1}{1 + (1-\varepsilon) x}$] () at (0.7,0.9) ;
  \draw[color=black,line width=2pt,domain=-0.42:1]    plot (\x,{(\x- 1)^2});
  \coordinate [label=$s$] () at (0.4,0.1) ;
\end{tikzpicture}
\caption{Constraints on SOS polynomial $s$ (in bold).}
\label{fig:ineq}
\end{center}
\end{figure}
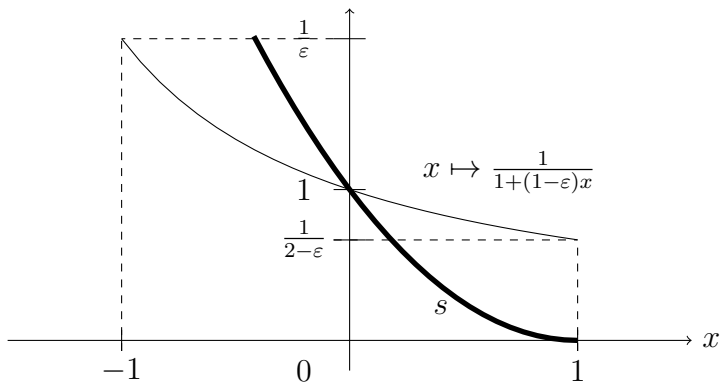

Assume first that $x \neq 0$. 
From \eqref{ineq}, if $x \in (0,1]$ then $s(x) \leq 1/(1+(1-\varepsilon)x)$, and if $x \in [-1,0)$ then $s(x) \geq 1/(1+(1-\varepsilon)x)$. 
By continuity of $s$, this implies that $s(0)=1$. If $x \in [0,1]$ then $1+(1-\varepsilon)x \in [1,2-\varepsilon]$ and $1/(1+(1-\varepsilon)x) \in [1/(2-\varepsilon),1]$. If $x \in [-1,0]$ then $1+(1-\varepsilon)x \in [\varepsilon,1]$ and $1/(1+(1-\varepsilon)x) \in [1,1/\varepsilon]$. 
It follows that $s$ should be above resp. below the hyperbola $x \mapsto 1/(1+(1-\varepsilon)x)$ for $x$ negative resp. positive, and in particular $s(-1) \geq 1/\varepsilon$, see Figure \ref{fig:ineq}.

A particular choice is
\[
s(x) = (ax-1)^{2d}
\]
for appropriate parameters $a \in \R$ and $d \in \N$. 
{An anonymous referee noticed that this polynomial is surprisingly similar to the one appearing in \cite[Eq.~(11)]{ns07}. 
}
In particular one should have $s(-1)=(a+1)^{2d} \geq 1/\varepsilon$ 
and $s(1)=(a-1)^{2d} \leq 1/(2-\varepsilon)$.  

\begin{proposition}
\label{prop:Paulynomial}
Let $\varepsilon \in (0,1]$, $a \in [1, 2)$ and $d \in \N$ be the {smallest integer} greater than 
$\frac{\log (\varepsilon^{-1})}{2 \log (1+a)}$ and 
$\frac{\log ((2 - \varepsilon)^{-1}) }{2 \log (a - 1)}$. 
Then the polynomial $s = (ax-1)^{2d}$ satisfies \eqref{ineq} and the relaxation \eqref{eq:SOS_d} of order $d+1$ is exact. 
\end{proposition}
\begin{proof}
Let $z:=1-\varepsilon$. 
To prove the claim, it is enough to show that $s$ should be above resp. below the hyperbola $x \mapsto 1/(1+zx)$ on the interval $[-1, 0]$ resp. $[0, 1]$, or equivalently that the polynomial function $x \mapsto f(x) = (ax-1)^{2d} (1+zx) - 1$ is above resp. below the $x$-axis on $[-1, 0]$ resp. $[0, 1]$. 
The derivative of $f$ is given by {
\begin{align*}
f'(x) & = 2d a (ax-1)^{2d-1} (1+zx) + z (ax-1)^{2d} \\
& = (ax-1)^{2d-1} [za(2d+1)x + 2da - z] \,.
\end{align*}
The first factor is below resp. above the $x$-axis on $[-1, a^{-1}]$ resp. $[a^{-1}, 1]$. 
The second factor is below resp. above the $x$-axis for $x \leq \frac{z-2da}{za(2d+1)}$ resp. $[\frac{z-2da}{za(2d+1)}, 1]$. 
Since $f(0)=0$ and $0 \in [\frac{z-2da}{za(2d+1)}, a^{-1}]$, one just needs to verify the sign of the values of $f$ at $-1$ and $1$. 
One has $f(-1) = (1+a)^{2d} (1 - z) - 1 = \varepsilon (1+a)^{2d} - 1$ and $f(1) = (a-1)^{2d} (1+z) - 1 = (a-1)^{2d} (2-\varepsilon) - 1$, thus by assumption one has $f(-1) \ge 0$ and $f(1) \le 0$, which yields the desired claim. }
\end{proof}
{
Proposition \ref{prop:Paulynomial} immediately implies the upper bound estimate from Theorem \ref{th:bounds} for the choice $a=1$ where $d = \lceil \frac{\log (\varepsilon^{-1})}{2 \log 2} \rceil $, yielding that $\varepsilon_d \leq 4^{-d+1}$ and thus $\varepsilon_{d+1} \leq 4^{-d}$. 
}
\begin{remark}
If one assumes that $\varepsilon \in \mathbb{Q}$, $\varepsilon>0$ then there is an algorithm computing an exact SOS decomposition~\eqref{sos0} with rational coefficients in  $\tilde{O}\left(\log(\frac{1}{\varepsilon})^4\right)$ {boolean operations}\footnote{$\tilde{O}(\cdot)$ stands for the variant of the big-$O$ notation ignoring logarithmic factors.}. 
Indeed from Proposition \ref{prop:Paulynomial} let us select $a=1$, $s = (x-1)^{2d}$ with $2 d = O\left( \log(\varepsilon^{-1}) \right)$, so that $t : x \mapsto x - s(x) (x + (1-\varepsilon) x^2)$ is nonnegative on $[-1, 1]$. 
This nonnegativity condition is equivalent to nonnegativity of $x \mapsto (1+x^2)^{2 d} t\left(\frac{x^2-1}{1 + x^2}\right)$ on $\R$. 
This latter polynomial has coefficients with maximal bit size upper bounded by $\tilde{O}(d)$. 
As a consequence of \cite[Theorem~24]{mag19}, there exists an algorithm  computing an exact SOS decomposition of $(1+x^2)^{2 d} t\left(\frac{x^2-1}{1 + x^2}\right)$ with rational coefficients in $\tilde{O}(d^4)$ {boolean operations}. 
\end{remark}

We now focus on the {lower} bound estimate. 
Let $\varepsilon \in (0,1]$ and $s \in \Sigma[x]_{2d}$  satisfying~\eqref{ineq}.
As discussed above, the following properties hold:

\begin{equation}\label{as:poly_s}
	\begin{cases}
	s(0) = 1 \,, \\
	s(x) \leq 1 \,, \forall x \in [0,1] \,,\\
	s(x) \geq 0 \,, \forall x \in \RR \,. \\
	\end{cases}
\end{equation}

We denote by $S_d$ the set of polynomials of degree $2d$
satisfying~\eqref{as:poly_s}.

\begin{proposition}
\label{prop:max_finite}
For any $d \in \NN$, $\max_{s \in S_d} s(-1)$ is  upper bounded by $1 + 2 d (4 e)^{2d}$.
\end{proposition}
\begin{proof}
From \cite[Lemma~4.1]{sh12}, one  has  $|s_k| \leq (4 e)^{2d} \displaystyle\max_{0 \leq j \leq 2 d} s\left(\frac{j}{2d}\right) \leq (4 e)^{2d}$, for all $k \in \{0,\dots,2d\}$ as $s$ is upper bounded by $1$ on $[0, 1]$. 
Thus, $s(-1) = \displaystyle\sum_{k=0}^{2d} (-1)^k s_k \leq 1 + 2 d (4 e)^{2d}$. 
\end{proof}
Proposition \ref{prop:max_finite} immediately implies the {lower} bound estimate from Theorem \ref{th:bounds} since any $s \in \Sigma[x]_{2 d}$ satisfying \eqref{ineq} must in particular satisfy $s(-1) \geq 1/\varepsilon$. \\
\if{
The final result of this section provides a lower bound on the moment-SOS relaxation value at order $d+1$. 
\begin{lemma}
\label{lemma:cvgrate}
For all $\varepsilon \in [0, 1]$ and $d \in \N$, one has $v_{d+1}(\varepsilon) \ge \min \{0, 2^{2d} \varepsilon - 1 \}$. 
\end{lemma}
\begin{proof}
When $\varepsilon = 0$, one has $v_1(0) = v^\star(0) = -1$ as stated in {Proposition} \ref{prop:finitecvg}. 
Let us assume that $\varepsilon \in (0, 1]$. 
By fixing $s = (x-1)^{2d}$ in \eqref{eq:SOS_d}, one has 
\begin{align*}
v_{d+1}(\varepsilon) \ge \sup & \ v \\
\text{s.t.} & \ x - v  - (x-1)^{2d} (x + z x^2) \ge 0 \,, \quad  \forall x \in [-1,1] \,.
\end{align*}
Hence one has $v_{d+1}(\varepsilon) \geq \min_{x \in [-1,1]} g(x)$ with $g(x) = x [1 - (x-1)^{2d} (1+z x) ] $. 
To prove the desired statement, we will show that $\min\{0, 2^{2d} \varepsilon - 1\}$ is a lower bound of the minimum of $g$ over $[-1, 1]$. 
Let us study the variations of $h(x) = (x-1)^{2d} (1+z x)$. 
Its derivative is  $h'(x) = (x-1)^{2d-1} ((2d+1) z x + 2 d -z )$,  which vanishes at $1$ and $x_d = \frac{z-2d}{z(2d+1)}$. 
Therefore $h$ is nondecreasing on $[-1, x_d]$ and nonincreasing on $[x_d, 1]$. In addition $h(-1) = 2^{2d} \varepsilon$, $h(0) = 1$ and $h(1) = 0$. 
It follows that  $g = x (1-h)$ is nonnegative on $[0, 1]$. 
If $d$ is such that $2^{2d} \varepsilon \leq 1$ then the function $1 - h$ is nonnegative on $[-1, 0]$, thus $g$ is greater than $2^{2d}\varepsilon - 1$ on $[-1, 0]$.  
Otherwise, $g$ is nonnegative on $[-1, 0]$. 
\end{proof}
}\fi

{
We end this section by additional numerical experiments suggested by an anonymous referee.  
The optimal value of the quasi-convex SOS program \eqref{eq:varepsilon_d} can be approximated as closely as desired by applying a suitable bisection scheme on $\varepsilon$.  
Doing so, we provide interval enclosures of the exact value of $\varepsilon_d$, for integer values of $d$ between $2$ and $10$. 
The corresponding results are displayed in Figure \ref{fig:epsd}. 
The dotted curves correspond to the lower and upper bounds derived in Theorem \ref{th:bounds}. 
\begin{figure}[H]
\begin{center}
\begin{tikzpicture}[xscale=.5,yscale=.1]
\draw[->] (-0.2, 0) -- (11, 0) node[right] {$d$};
\draw[->] (0, -0.2) -- (0, 50) node[above] {$\log \varepsilon_d^{-1}$};
\draw plot[mark=x,smooth] coordinates  {(2, 2.0101) (3, 5.357) (4,  8.873) (5, 12.399) (6, 15.9244) (7, 19.449) (8, 22.9754) (9, 26.500) (10, 30.0264) };
\draw[dotted,line width=0.7mm] plot coordinates  {(2, 1.3862943611198906) (3,  2.772588722239781) (4,  4.1588830833596715) (5,  5.545177444479562) (6,  6.931471805599453) (7,  8.317766166719343) (8,  9.704060527839234) (9, 11.090354888959125) (10, 12.476649250079015)};
\draw[dotted, line width=0.7mm] plot coordinates  {(2, 5.46995621235328) (3, 10.931489691805597) (4, 16.109525736808077) (5, 21.169796431278805) (6, 26.165528704197282) (7, 31.120438983226716) (8, 36.047178385293726) (9, 40.95329850015803) (10, 45.843670258054196)};
\draw (2,-2.5) node {$2$};  \draw (2,10pt) -- (2,-10pt);
\draw (3,10pt) -- (3,-10pt);
\draw (4,10pt) -- (4,-10pt);
\draw (5,10pt) -- (5,-10pt);
\draw (6,10pt) -- (6,-10pt);
\draw (7,10pt) -- (7,-10pt);
\draw (8,10pt) -- (8,-10pt);
\draw (9,10pt) -- (9,-10pt);
\draw (10,-2.5) node {$10$}; 
\draw (10,10pt) -- (10,-10pt);
\draw (-0.6,0) node {$0$};  \draw (2pt,0) -- (-2pt,0);
\draw (-0.6,10) node {$10$};  \draw (2pt,10) -- (-2pt,10);
\draw (-0.6,20) node {$20$};  \draw (2pt,20) -- (-2pt,20);
\draw (-0.6,30) node {$30$};  \draw (2pt,30) -- (-2pt,30);
\draw (-0.6,40) node {$40$};  \draw (2pt,40) -- (-2pt,40);
\end{tikzpicture}
\caption{{{Values of $\log \varepsilon_d^{-1}$ obtained by solving  \eqref{eq:varepsilon_d} for values of $d \in \{2,\dots,10\}$. 
}}}
\label{fig:epsd}
\end{center}
\end{figure}
In particular we retrieve an interval enclosure  of $\varepsilon_2 \in [0.1339740, 0.1339750]$ whose exact value, given in Lemma \ref{lemma:eps2}, is $1 - \sqrt{3}/2 \approx 0.1339745$.
We could also numerically certify that $\varepsilon_3 \in [4.712522 \cdot 10^{-3}, 4.7125518 \cdot 10^{-3}]$, which shows that the theoretical upper bound $4.712527\cdot 10^{-3}$ provided in Lemma \ref{lemma:eps3} is actually very close to the true value of $\varepsilon_3$. 
The function $d \mapsto \log \varepsilon_d^{-1}$ has an approximate linear growth rate of 3.5, between the lower and upper bounds from Theorem \ref{th:bounds}. 
The Julia source code implementing these experiments is available online.\footnote{\url{https://homepages.laas.fr/vmagron/files/slowuniv.jl}}
}
{\section{Discussions and further remarks}}
%
In this note we studied a specific SOS decomposition of a parametrized univariate polynomial to design an elementary POP for which the moment-SOS hierarchy shows finite yet arbitrarily slow convergence when the parameter tends to a limit.

{
As suggested by an anonymous referee, we end this section with a discussion about the finite convergence behavior of the moment-SOS hierarchy, related to the saturation and stability properties of the considered quadratic module. 
By definition of the moment-SOS hierarchy, if the quadratic module $\cM(S)$ is saturated then the moment-SOS hierarchy has finite convergence, and $\sup=\max$ is attained in the SOS part, when we reach the finite convergence. 
If $\cM(S)$ is stable, then the moment-SOS hierarchy stabilizes. 
These results can be found in, e.g., \cite[Chapter~10]{marshall08}.

\paragraph{{Saturated and stable quadratic module}}
The univariate case with nonempty interior is essentially the unique case, where these two properties can happen simultaneously, and we provide such an example in our study. 
In particular, the quadratic module $\cM(S_{\varepsilon})$ of problem \eqref{univpara}  is both saturated and stable, for $\varepsilon >0$. 
Saturation has been shown in the proof of {Proposition}~\ref{prop:finitecvg}. 
In addition, one can show that $\cM(S_{\varepsilon})$ is stable by choosing a so-called \textit{natural description} for $K(S_{\varepsilon}) = [0, 1]$ with the generator $h=x(1-x)$. 
For more details about natural descriptions see \cite{marshall08}, especially Proposition~2.7.3 and its proof. 
This notion of stability depends only on the quadratic module, and not on the particular finite system of generators. 
Let us consider a polynomial $p \in \cM(h)$ of degree $d$. 
Since $\cM(S_{\varepsilon})$ is saturated, $p$ is nonnegative on $K(\{h\})$, so admits by \cite{fekete35} a representation $p = \sigma_0 + \sigma_1 h $ with SOS polynomials $\sigma_0 $ of degree $2 \lceil d / 2 \rceil$ and $\sigma_1$ of degree $2 \lceil d / 2 \rceil - 2$.
So every polynomial of degree $d$ in $\cM(\{h\})$ can be represented in $\cM(\{h\})_{2k}$ with $k=\lceil d / 2 \rceil$, thus $\cM({h})=\cM(S_{\varepsilon})$ is stable.
%
%
Moreover, if we fix a univariate quadratic module as above (stable+saturated) and let the objective function vary (keeping the degree bounded) we cannot observe that the order of finite convergence goes to infinity, as this would contradict the stability of the quadratic module.

\if{
In \cite{pow00}, the authors focus on finding similar SOS decompositions associated to the minimization of the polynomial $1+x+\varepsilon$ (parametrized by $\varepsilon > 0$) on $K(S)=[-1, 1]$, where the latter set is encoded as the super-level set of an {odd} power $r$ of $1 -x^2$. 
They prove in \cite[Theorem~7]{pow00} that finite convergence is reached after a number of iterations  proportional to the power $r$ and the reciprocal $1/\varepsilon$ of the parameter. 
In this case, for a given odd value of $r \geq 3$ the semialgebraic set is fixed and the associated  quadratic module is neither saturated nor stable. 
The former is due to the fact that $1+x$ is nonnegative on $K(S)$ but does not belong to the quadratic module $\cM(\{(1-x^2)^r\})$. 
The latter is due to the fact that the family of linear polynomials $(1+x+\varepsilon)_{\revise{\varepsilon} > 0}$ is contained in the module but {has} arbitrary large SOS representations. 
}\fi
\paragraph{Neither saturated nor stable quadratic module}
{In \cite{st96}, Stengle focuses on finding similar SOS decompositions associated to the minimization of the polynomial $1-x^2+\varepsilon$ (parametrized by $\varepsilon > 0$) on $K(S)=[-1, 1]$, where the latter set is encoded as the super-level set of of $(1 -x^2)^3$. 
    Stengle proves in \cite[Theorem~4]{st96} that finite convergence is reached after a number of iterations proportional to the reciprocal square root of the parameter, i.e. $1/\sqrt{\varepsilon}$.
In this case the semialgebraic set is fixed and the associated  quadratic module is neither saturated nor stable. 
The lack of saturation is due to the fact that $1-x^2$ is nonnegative on $K(S)$ but it does not belong to the quadratic module $\cM(\{(1-x^2)^3\})$. 
The lack of stability is due to the fact that the family of  polynomials $(1-x^2+\varepsilon)_{\revise{\varepsilon} > 0}$ is contained in the module but it has arbitrary large SOS representations. }

\paragraph{{Saturated not stable quadratic module}}
As suggested by an anonymous referee the same phenomenon can happen with a saturated quadratic module that is not stable.
To observe the order of finite convergence going to infinity while keeping the semialgebraic set fixed, with an associated quadratic module being saturated, it is necessary to go one dimension higher. 
For instance, Example 3.7 in \cite{netzer10} is a candidate.  
Indeed, let us consider the preordering, i.e., the quadratic module generated by all the products of $S=\{p_1, p_2, p_3, p_4\}$ as in the example, and as objective functions the linear polynomials  $l_a = a_0 + a_1 t_1 + a_2 t_2  \in \R[t_1, t_2]$ parametrized by $a \in \R$ such that $l_a$ is nonnegative on $K(S)$. 
In this case the preordering is saturated, and thus the moment-SOS hierarchy has finite convergence for every objective polynomial. 
But the order of finite convergence is going to infinity, as the family of linear polynomials has representations of unbounded degree by \cite[Theorem~3.5]{netzer10}. 
Moreover, since the semialgebraic set has nonempty interior, $\sup=\max$ in the SOS hierarchy at every order. 
}
%

\paragraph{{Extension to the multivariate case}}
In our case Theorem \ref{th:bounds} shows that finite convergence is reached  
after a number of iterations proportional to the parameter bit size inverse $\log(1/\varepsilon)$. 
Overall the results from {\cite{st96}} imply that for univariate POPs the minimal relaxation order required for finite convergence could be exponential in the bit size of the coefficients involved in the objective function. {We also point out \cite[Theorem~7]{pow00} where the authors consider SOS representations associated to the minimization of $1+x+\varepsilon$ on $K(S)=[-1, 1]$, where the latter set is encoded as the super-level set of an odd power of $1 -x^2$. 
If the power is greater than 3 then they provide upper bounds on the minimal order that are proportional to the reciprocal of the parameter, i.e. $1/\varepsilon$.}
Our complementary study shows that the minimal order can be linear in the bit size of the coefficients involved in the constraints. 
{
Further research could focus on studying quantitatively higher-dimensional cases where the finite convergence order goes to infinity while keeping the semialgebraic set fixed; as suggested by an anonymous referee  \cite[Example 3.7]{netzer10} could be a candidate. 
}
{
The same behavior happens for the following (contrived) multivariate problem built from the univariate one from our paper, e.g., when minimizing the polynomial $x_1 + \dots + x_n$ on the set of constraints $1 - x_i^2 \geq 0, x_i + (1-\varepsilon) x_i^2 \geq 0, i=1,\dots,n$. 
Similar results hold for the bivariate polynomial optimization problem \eqref{poppara} from Section \ref{sec:poppara}, since any SOS representation associated to \eqref{univpara} provides one for \eqref{poppara}. 
Indeed suppose that there exists $d \in \N$ such that such that $x_1  \in \cM(S_{\varepsilon})_{2d}$, that is, $x_1  = q + r (1-x_1^2) + s(x_1 + (1 - \varepsilon) x_1^2)$ for some $q \in \Sigma[x_1]_{2 d}$, and $r,s \in \Sigma[x_1]_{2(d-1)}$. 
Then $x_1  = q + r x_2^2 + (r - (1 - \varepsilon)s) (1 - x_1^2 - x_2^2) + s (1 - \varepsilon + x_1 - (1-\varepsilon) x_2^2)$, so $x_1  \in \cM(S^2_{\varepsilon})_{2d}$ with $S^2_{\varepsilon} = \{\pm 1-x_1^2-x_2^2,  1 - \varepsilon + x_1 - (1-\varepsilon) x_2^2\}$. 
}

\paragraph{{Influence of the number of connected components and solver precision}}
As in {\cite{st96}}, the description of the set of constraints plays a crucial role in the behavior of the moment-SOS hierarchy. 
For $\varepsilon \in (0, 1]$, our considered set of constraints is the interval $[0, 1]$. 
When this interval is described by either the super-level set of the single quadratic polynomial $x(1-x)$ or by the intersection of the ones of $x$ and $1 -x$, the hierarchy immediately converges at order $1$. 
The present work shows that a slight modification of this description, by means of only two quadratic polynomials, can have significant impact on the efficiency of the moment-SOS hierarchy. 
{
What happens here is that the interval defined by the first constraint almost intersects one side of the second constraint (as $\varepsilon$ goes to zero). 
One might think that the convergence behavior tends to be bad when the semialgebraic sets defined by constraints almost intersect. 
Further research efforts should be pursued to investigate how bad the convergence behavior could be when the number of connected components of the feasible set grows.}\\
In addition to slow convergence behaviors, complementary studies such as \cite{wa12} have emphasized that inappropriate constraint descriptions may provide wrong relaxation bounds when relying on  standard double-precision semidefinite solvers. 
A possible remedy consists of using instead a multiple precision semidefinite solver, which, again, might slow down significantly the required computation. \\
Further investigation should focus on means to  appropriately describe POP constraint sets, in order to obtain sound numerical solutions at lower relaxation orders and with lower solver accuracy. 

\section*{Acknowledgments}
The authors thank Pauline Kergus for suggesting the polynomial from  Section~\ref{sec:higher} {and the anonymous referees for their very insightful feedback. 
The authors are also grateful to Michal Ko{\v c}vara and Nando Leijenhorst for their advice regarding the use of the high-precision SDP solvers Loraine and ClusteredLowRankSolver,  respectively.
}
This work benefited from the FastOPF2 grant funded by RTE, the EPOQCS grant funded by the LabEx CIMI (ANR-11-LABX-0040), the HORIZON–MSCA-2023-DN-JD of the European Commission under the Grant Agreement No 101120296 (TENORS), the AI Interdisciplinary Institute ANITI funding, through the French ``Investing for the Future PIA3'' program under the Grant agreement n${}^\circ$ ANR-19-PI3A-0004 as well as the National Research Foundation, Prime Minister’s Office, Singapore under its Campus for Research Excellence and Technological Enterprise (CREATE) programme.
This work was performed within the project
COMPUTE, funded within the QuantERA II Programme that has received funding from the EU’s H2020 research and innovation programme under the GA No 101017733. 
This work was co-funded by the European Union under the project ROBOPROX (reg. no. CZ.02.01.01/00/22 008/0004590). 

%


\bibliographystyle{plain}

\begin{thebibliography}{99}
{\bibitem{aa20}
E. D. Andersen, K. D. Andersen. 
The MOSEK interior point optimizer for linear programming: an implementation of the homogeneous algorithm. 
In H. Frenk, K. Roos, T. Terlaky, S. Zhang (Editors). High Performance Optimization, 197-232. Springer, 2000.}
\bibitem{b22}
L. Baldi. Repr\'esentations effectives en g\'eom\'etrie alg\'ebrique r\'eelle
et optimisation polynomiale. PhD thesis, Inria Univ. C\^ote d'Azur, 2022.
\bibitem{bs23}
L. Baldi, L. Slot.
Degree bounds for Putinar's Positivstellensatz on the hypercube. SIAM J. Applied Algebra and Geometry 8(1):1-25, 2024.
\bibitem{chancelier2021hidden}
J.P. Chancelier, M. De Lara.
Hidden convexity in the $l_0$ pseudonorm.
J. Convex Analysis
28(1):203-236, 2021.
\bibitem{fekete35} 
M. Fekete. 
Proof of three propositions of Paley. 
Bull. Amer. Math. Soc. 41:138–144, 1935.
\bibitem{gm12}
B. G{\"a}rtner, J. Matousek. Approximation algorithms and semidefinite programming. Springer Science \& Business Media, 2012.	
\bibitem{hks23}
S. Habibi, M. Ko{\v{c}}vara, M. Stingl. 
Loraine - an interior-point solver for low-rank semidefinite programming. 
Optimization Methods and Software, 1-31, 2023.
\bibitem{h23}
D. Henrion. Moments for polynomial optimization - An illustrated tutorial. Lecture notes of a course given for the programme Recent Trends in Computer Algebra, Institut Henri Poincar\'e, Paris, 2023.
\bibitem{hkl20}
D. Henrion, M. Korda, J. B. Lasserre.
The moment-SOS hierarchy.
World Scientific, 2020.
\bibitem{hl03}
D. Henrion, J. B. Lasserre. Solving global optimization problems over polynomials with GloptiPoly 2.1. LNCS 2861:43-58, Springer, 2003.
\bibitem{jh16}
C. Josz, D. Henrion. Strong duality in Lasserre's hierarchy for polynomial optimization. Optim. Letters 1(10):3-10, 2016.
\bibitem{l01}
J. B. Lasserre.
Global optimization with polynomials and the problem of moments. SIAM J. Optimization 11(3):796--817, 2001.
\bibitem{l09}
J. B. Lasserre. 
Moments, positive polynomials and their applications. World Scientific, 2009. 
\bibitem{le2022}
A. Le Franc, J.P. Chancelier, M. De Lara.
The Capra-subdifferential of the $l_0$ pseudonorm.
Optimization 73(4):1229-1251, 2022.
\bibitem{ll22}
N. Leijenhorst, D. de Laat. Solving clustered low-rank semidefinite programs arising from polynomial optimization. {Math. Prog. Comput. 16:503–534, 2024.}
\bibitem{mag19}
V. Magron, M. Safey El Din, M. Schweighofer. Algorithms for weighted sum of squares decomposition of non-negative univariate polynomials. J. Symbol. Comput. 93:200-220, 2019. 
\bibitem{mag22}
V. Magron, M. Safey El Din, M. Schweighofer, T. H. Vu. 
Exact SOHS decompositions of trigonometric univariate polynomials with Gaussian coefficients. 
Proc. Int. Symp. on Symbolic and Algebraic Comput. (ISSAC), 325-332, 2022.
{\bibitem{marshall08}
M. Marshall. 
Positive polynomials and sums of squares. 
American Mathematical Society, 2008.}
{\bibitem{netzer10}
T. Netzer, D. Plaumann, M. Schweighofer. 
Exposed faces of semidefinitely representable sets. 
SIAM Journal on Optimization, 20(4):1944-1955, 2010.}
\bibitem{nie2014optimality}
J. Nie. Optimality conditions and finite convergence of Lasserre’s hierarchy. Math. Prog. 146:97-121, 2014.
\bibitem{n23}
J. Nie. Moments and polynomial optimization. SIAM, 2023.
\bibitem{ns07}
J.~Nie and M.~Schweighofer.
On the complexity of Putinar's Positivstellensatz. 
Journal of Complexity, 23(1):135-150, 2007.
\bibitem{pow00}
V. Powers and B. Reznick. Polynomials that are positive on an interval. Trans. Amer. Math. Soc., 352(10):4677-4692, 2000. 
\bibitem{ps76}
G. P\'olya and G. Szeg\"o. Problems and theorems in analysis II, Springer, 1976.
\bibitem{sh12}
A. A. Sherstov. Making polynomials robust to noise. Proceedings of the ACM Symposium on Theory of Computing, 747-758, 2012. 
\bibitem{s22}
L. Slot. Asymptotic analysis of semidefinite bounds for polynomial optimization and independent sets in geometric hypergraphs. PhD thesis, CWI Amsterdam and Tilburg Univ., 2022.
\bibitem{st96}
G. Stengle. Complexity estimates for the Schm\"udgen Positivstellensatz, Journal of Complexity, 12:167-174, 1996.
\bibitem{sturm99}
J. F. Sturm. 
Using SeDuMi 1.02, a MATLAB toolbox for optimization over symmetric cones. 
Optimization Methods and Software, 11(1-4):625-653, 1999.
\bibitem{wa12}
H. Waki, M. Nakata and M. Muramatsu. 
Strange behaviors of interior-point methods for solving semidefinite programming problems in polynomial optimization. 
Computational Optimization and Applications, 53:823-844, 2012.
\end{thebibliography}

\end{document}